\documentclass[11pt]{article}
\usepackage{amscd}
\usepackage{mathrsfs}
\usepackage{a4wide}
\usepackage{color}
\usepackage{amssymb,amsmath,amsthm,amsfonts,latexsym}
\usepackage[utf8x]{inputenc}
\usepackage{tikz-cd}

\usepackage{amsmath,amsfonts,amssymb,amsthm,bbm} 
\usepackage{tikz}
\usepackage{hyperref}
\usepackage[all]{xy}
\usepackage[nobysame,initials]{amsrefs}
\usepackage{ifthen,latexsym,amssymb,amsmath,bbm,fixmath}
\usepackage[shortlabels]{enumitem}

\newcommand{\C}[1]{{\protect\mathcal{#1}}}
\newcommand{\B}[1]{{\mathbold #1}}
\newcommand{\I}[1]{{\mathbbm #1}}

\newcommand{\V}[1]{\mathbold{#1}}

\newcommand{\e}{\varepsilon}
\newcommand{\floor}[1]{\lfloor #1\rfloor}
\newcommand{\me}{{\mathrm e}}
\renewcommand{\mid}{:}

\renewcommand{\ge}{\geqslant}
\renewcommand{\le}{\leqslant}

\newif\ifnotesw\noteswtrue

\newcommand{\hide}[1]{}

\newcommand{\beq}[1]{\begin{equation}\label{#1}}
	\newcommand{\eeq}{\end{equation}}

\newtheorem{theorem}{Theorem}[section]
\newtheorem{corollary}[theorem]{Corollary}
\newtheorem{lemma}[theorem]{Lemma}
\newtheorem*{claim*}{Claim}

\theoremstyle{definition}
\newtheorem{definition}[theorem]{Definition}

\theoremstyle{remark}
\newtheorem{remark}[theorem]{Remark}

\newcommand{\tP}{\widetilde{\mathbb{P}}}

\newcommand{\tW}{\widetilde{\mathcal{W}}}

\def\leukfrac#1/#2{\leavevmode
               \kern.1em
                \raise.9ex\hbox{\the\scriptfont0 ${}_#1$}
                \hskip -1pt\kern-.1em
                /\kern-.15em\lower.10ex\hbox{\the\scriptfont0 ${}_#2$}}

\makeatletter
\def\diam{\mathop{\operator@font diam}\nolimits}
\makeatother

\newcommand{\DZ}[1]{\cite[#1]{DemboZeitouni10ldta}}

\newcommand{\RS}[1]{\cite[#1]{RassoulaghaSeppelainen14cldigm}}

\newcommand{\Lo}[1]{\cite[#1]{Lovasz:lngl}}

\newcommand{\Ak}{{\bf A}^{(k)}}
\newcommand{\Wk}{\mathcal{W}^{(k)}}
\newcommand{\tWk}{\widetilde{\mathcal{W}}^{(k)}}
\newcommand{\Walpha}{\mathcal{W}\times \Aalpha}
\newcommand{\tWalpha}{\widetilde{\mathcal{W}}^{(\V\alpha)}}

\newcommand{\T}[1]{\widetilde{#1}}
\newcommand{\tR}{\widetilde{\mathbb{R}}}

\newcommand{\dd}{\,\mathrm{d}}
\newcommand{\Aalpha}{{\bf A}^{(\V\alpha)}}
\newcommand{\f}[1]{f^{#1}}
\newcommand{\tf}[1]{\T{f^{#1}}}

\newcommand{\Sball}{S}

\newcommand{\cI}[2]{I_{#2}^{(#1)}}

\newcommand{\NoZero}{}
\newcommand{\NZ}{\I N_{\ge 0}}
\newcommand{\Tk}[1]{\widetilde{(#1)}}
\renewcommand{\rho}{r}
\begin{document}

\title{Large Deviation Principles for Block and Step Graphon Random Graph Models}
\author{Jan Greb\'\i k\\
	Mathematics Institute\\
	University of Warwick\\
	Coventry CV4 7AL, UK \and 
	Oleg Pikhurko\\
	Mathematics Institute and DIMAP\\
	University of Warwick\\
	Coventry CV4 7AL, UK}

\maketitle

\noindent{\large\bf This preprint is fully superseded by arXiv:2311.06531, as 
its all main results are included (with the same proofs) into the latter.}

\begin{abstract}
 Borgs, Chayes, Gaudio, Petti and Sen~\cite{BCGPS} proved a large deviation principle for 
 block model random graphs with rational block ratios. We strengthen their result by allowing any block ratios (and also establish a simpler formula for the rate function). We apply the new result to derive a large deviation principle for graph sampling from any given step graphon.
\end{abstract}

\section{Introduction}

The theory of large deviations (see e.g.\ 
the books by Dembo and Zeitouni~\cite{DemboZeitouni10ldta} or Rassoul-Agha and Sepp\"al\"ainen~\cite{RassoulaghaSeppelainen14cldigm})
% books~\cite{DemboZeitouni10ldta,DemboZeitouni10ldta}
 studies the probabilities of rare events on the exponential scale. 
This is formally captured by the following definition.

\begin{definition}\label{df:LDP}
	A sequence of Borel probability measures $(\mu_n)_{n\in \I N}$ on a 
	topological space $X$ satisfies a \emph{large deviation principle} (\emph{LDP} for short) with \emph{speed} 
	$s:\I N\to(0,\infty)$
	and \emph{rate function} $I:X\to[0,\infty]$ if
	\begin{itemize}[nosep]
		\item 
	$s(n)\to\infty$ as $n\to\infty$, 
	\item the function $I$ is \emph{lower semi-continuous} (\emph{lsc} for short), that is, for each $\rho\in\I R$
	the level set $\{I\le \rho\}:=\{x\in X\mid I(x)\le \rho\}$ is a closed subset of~$X$,
	\item the following \emph{lower bound} holds:
	\beq{eq:lowerGen}
		\liminf_{n\to\infty} \frac1{s(n)}\,{\log\, \mu_n(G)} 
		\ge -\inf_{x\in G} I(x),\quad \mbox{for every open $G\subseteq X$,}
		\eeq
		\item the following \emph{upper bound} holds:
		\beq{eq:upperGen}\limsup_{n\to\infty} \frac1{s(n)}\,{ \log\,\mu_n(F)} \le  -\inf_{x\in F} I(x),\quad \mbox{for every closed $F\subseteq X$.}
\eeq
\end{itemize}
\end{definition}

As it is well-known (see e.g.\ \RS{Lemma~2.11}), if~\eqref{eq:lowerGen} and~\eqref{eq:upperGen} hold for some (not necessarily lsc) function $I:X\to [0,\infty]$ then we can replace $I$ without violating these bounds by its \emph{lower semi-continuous
	regularization} 
\beq{eq:lscR}
 I_{\mathrm{lsc}}(x):=\sup\left\{\inf_{y\in G} I(y)\mid G\ni x\mbox{ and $G\subseteq X$ is open}\right\},\quad x\in X;
\eeq 
furthermore (see e.g.\ \RS{Lemma~2.8}), 
 $I_{\mathrm{lsc}}$ is lower semi-continuous  and, in fact, $I_{\mathrm{lsc}}$ is the largest lsc function with $I_{\mathrm{lsc}}\le I$. 
%Thus if $I$ is lower semi-continuous then $I_{\mathrm{lsc}}=I$. 
If $X$ is a regular topological space then there can be at most one lower semi-continuous rate function satisfying Definition~\ref{df:LDP} (see e.g.\ \DZ{Lemma 4.1.4} or \RS{Theorem~2.13}). This (as well as some other results, such as Lemma~\ref{lm:LDP} below) motivates the requirement that $I$ is lsc in Definition~\ref{df:LDP}.

Large deviations for various models of random graphs have been receiving much attention in the recent years; see e.g.\ the survey by Chatterjee~\cite{Chatterjee16bams}, or \cite[Section~1.7]{BCGPS} for references to some more recent results. A basic but central model
is the \emph{binomial random graph} $\I G(n,p)$, where the vertex set is $[n]:=\{1,\dots,n\}$ and 
each pair of vertices is an edge with probability~$p$, independently of other pairs. A large deviation principle for $\I G(n,p)$ for constant $p\in (0,1)$ was established in a ground-breaking paper of
Chatterjee and Varadhan~\cite{ChatterjeeVaradhan11} as follows. (See also the exposition of this proof in  Chatterjee's book~\cite{Chatterjee17ldrg}.) 

As it turns out, the ``correct'' setting is to consider \emph{graphons}, that is,  measurable symmetric functions $[0,1]^2\to [0,1]$. On the set $\C W$  of all {graphons}, one can define the so-called \emph{cut-distance} $\delta_\Box$, which is a pseudo-metric on $\C W$ (see Section~\ref{graphons} for all missing definitions related to graphons). Consider the factor space 
$$
\tW:=\{\T U: U\in\C W\},
$$ 
where $\T U:=\{V\in\C W\mid \delta_\Box(U,V)=0\}$ consists of all graphons \emph{weakly isomorphic} to~$U$.
 The space $(\tW,\delta_\Box)$ naturally appears in the limit theory of dense graphs, see e.g.\ the book by
Lov\'asz~\cite{Lovasz:lngl}. In particular, a graph $G$ on $[n]$ can be identified with the graphon $\f{G}$ where we partition $[0,1]$ into intervals of length $1/n$ each and let $\f{G}$ be the $\{0,1\}$-valued step function that encodes the adjacency relation.
This way, $\I G(n,p)$ gives a (discrete) probability measure $\tP_{n,p}$ on $(\tW,\delta_\Box)$, where $\tP_{n,p}(A)$ for $A\subseteq \tW$ is the probability that the sampled graph,
when viewed as a graphon up to weak isomorphism, belongs to the set~$A$. 

Also, recall that, for $p\in [0,1]$,  the \emph{relative entropy} $h_p:[0,1]\to [0,\infty]$ is defined by
\beq{eq:hp}
 h_p(\rho):=\rho\log\left(\frac{\rho}{p}\right)+(1-\rho)\log\left(\frac{1-\rho}{1-p}\right),\quad \rho\in[0,1].
\eeq

\begin{theorem}[Chatterjee and Varadhan~\cite{ChatterjeeVaradhan11}]
	\label{th:CV}
	Let $p\in [0,1]$. The function $I_p:\C W\to [0,\infty]$ defined by 
	\beq{eq:IpCV}
	 I_p(U):=\frac12 \int_{[0,1]^2} h_p(U(x,y))\dd x\dd y,\quad U\in\C W,
	 \eeq
	 gives a well-defined function $\tW\to [0,\infty]$ (that is, $I_p$ assumes the same value at any two graphons at $\delta_\Box$-distance 0) which is lower semi-continuous on $(\tW,\delta_\Box)$. Moreover,
	 the sequence of measures $(\tP_{n,p})_{n\in\I N}$ on $(\tW,\delta_\Box)$ satisfies an LDP with speed $n^2$ and rate function~$I_p$.
	 \end{theorem}

Borgs, Chayes, Gaudio, Petti and Sen~\cite{BCGPS} extended this result to $k$-block stochastic models as follows. 
Let $k\ge 1$ be a fixed integer. Let $\V p=(p_{i,j})_{i,j\in [k]}\in [0,1]^{k\times k}$ be a symmetric $k\times k$ matrix with entries in~$[0,1]$. For an integer vector $\V a=(a_1,\dots,a_k)\in\NZ^k$, where 
$$
\NZ:=\{0,1,2,\ldots\}
$$
 denotes the set of non-negative integers, let $\tP_{\V a,\V p}$ be the probability distribution on $\tW$ defined as follows.  Set $n$ to be $\|\V a\|_1=a_1+\dots+a_k$ and  let $(A_1,\dots, A_k)$ be the partition of $[n]$ into $k$ consecutive intervals with $A_i$ having $a_i$ elements. The random graph $\I G(\V a,\V p)$ on $[n]$ is produced by making each pair $\{x,y\}$ of $[n]$ an edge with probability $p_{i,j}$ where $i,j\in [k]$ are the indices with $x\in A_i$ and $y\in A_j$, with all choices made independently of each other. Output the weak isomorphism class $\tf{G}\in\tW$ of the graphon $\f{G}$ corresponding to the generated graph~$G\sim \I G(\V a,\V p)$ on~$[n]$. Informally speaking, we take $k$ blocks consisting of exactly $a_1,\dots,a_k$ vertices respectively, make pairs into edges with the probabilities given by the $k\times k$ matrix~$\V p$, and them forget the block structure.

Next, we define a rate function for a given non-zero real $k$-vector $\V\alpha=(\alpha_1,\dots,\alpha_k)\in[0,\infty)^k\NoZero$. 
 Let $(\cI{\V\alpha}{1},\dots,\cI{\V\alpha}{k})$ denote the partition of $[0,1]$ into consecutive intervals such that each interval $\cI{\V\alpha}{i}$ has length $\alpha_i/\|\V\alpha\|_1$. 
 Define the function $J_{\V\alpha,\V p}:\tW\to [0,\infty)$ by
 \beq{eq:JAlphaP}
J_{\V\alpha,\V p}(\T U):=\inf_{V\in \T U}\,  \frac12\sum_{i,j\in [k]} \int_{\cI{\V\alpha}{i}\times \cI{\V\alpha}{j}} h_{p_{i,j}}(V(x,y))\dd x\dd y,\quad U\in {\C W}.
 \eeq
 Note that the function $J_{\V\alpha,\V p}$ will not change if we multiply the vector $\V\alpha$ by any positive scalar.
 
In the above notation, the LDP of Borgs et al~\cite[Theorem 1 and Remark~2]{BCGPS} states the following.
 
 \begin{theorem}[Borgs, Chayes, Gaudio, Petti and Sen~\cite{BCGPS}]
 	\label{th:BCGPS}
 	Let $\V\alpha\in \NZ^k\NoZero$ be a non-zero integer $k$-vector and let $\V p\in [0,1]^{k\times k}$ be a symmetric $k\times k$ matrix. Then the sequence of measures $(\tP_{n\V\alpha,\V p})_{n\in\I N}$ on $(\tW,\delta_\Box)$ satisfies an LDP with speed $(n\,\|\V\alpha\|_1)^2$ and rate function~$(J_{\V\alpha,\V p})_{\mathrm{lsc}}$.
 	%, the lower semi-continuous regularization of $J_{\V\alpha,\V p}$ on $(\tW,\delta_\Box)$ as defined in~\eqref{eq:lscR}.
 	\end{theorem}
 
 Note that the special case $k=1$ and $\V\alpha=(1)$ of Theorem~\ref{th:BCGPS} and the assumption that
 the function $J_{\V\alpha,\V p}:\tW\to [0,\infty]$  is lower semi-continuous give the second part of Theorem~\ref{th:CV}.
 
 Our contribution is as follows.

First, we prove that the function $J_{\V\alpha,\V p}:\tW\to [0,\infty]$ is  lower semi-continuous (so, in particular, there is no need to take
 the lower semi-continuous
 regularization in Theorem~\ref{th:BCGPS}):
 
 \begin{theorem}\label{th:JLSC}
 For every symmetric matrix $\B p\in [0,1]^{k\times k}$ and every non-zero real $k$-vector $\V \alpha\in [0,\infty)^k$, the function $J_{\V\alpha,\V p}:\tW\to[0,\infty]$ is lower semi-continuous with respect to the metric~$\delta_\Box$.
 \end{theorem}

Second, we extend Theorem~\ref{th:BCGPS} by allowing the fraction of vertices assigned to a part to depend on $n$ as long as it converges to any finite (possibly irrational) limit.

\begin{theorem}\label{th:GenLDP}
	Fix 
	%an integer $k\ge 1$, 
	any symmetric $k\times k$ matrix $\V p\in [0,1]^{k\times k}$ and a non-zero real $k$-vector $\V\alpha=(\alpha_1,\dots,\alpha_k)\in [0,\infty)^k\NoZero$.  Let
	$$
	\V a_n=(a_{n,1},\dots,a_{n,k})\in \NZ^k\NoZero,\quad \mbox{for $n\in\I N$},
	$$ 
	be arbitrary non-zero integer $k$-vectors  such that $\lim_{n\to\infty} a_{n,i}/n=\alpha_i$ for each $i\in [k]$. 
	Then the sequence of measures $(\tP_{\V a_n,\V p})_{n\in\I N}$ on $(\tW,\delta_\Box)$ satisfies an LDP with speed $\|\V a_n\|_1^2$ and rate function~$J_{\V\alpha,\V p}$.
\end{theorem}

One application of Theorem~\ref{th:GenLDP}
is as follows. Each graphon $W\in\C W$ gives rise to the following inhomogeneous random graph model. Namely, the \emph{random $W$-graph} $\I G(n,W)$ is generated by first sampling
uniform elements $x_1,\dots,x_n\in [0,1]$ and then making each pair $\{i,j\}$ an edge with probability $W(x_i,x_j)$, where all choices are independent of each other. Let $\tR_{n,W}$ be the corresponding (discrete) measure on~$\tW$ where we take the equivalence class $\tf{G}\in\tW$ of the sampled graph~$G$. When $W$ is the constant function $p$, we get exactly the binomial random graph $\I G(n,p)$ and $\tR_{n,W}=\tP_{n,p}$.

The authors showed in~\cite{GrebikPikhurko:LDP} that, for any graphon $W\in\C W$, the only ``interesting'' speeds for the sequence of measures $(\tR_{n,W})_{n\in\I N}$ are $\Theta(n)$ and $\Theta(n^2)$, and established a general LDP for speed~$n$. The case when speed is $n^2$ seems rather difficult. Here (in Theorem~\ref{th:ourLDP}) we prove an LDP for speed $n^2$
when $W$ is a \emph{$k$-step graphon}, that is, there is a measurable partition $[0,1]=A_1\cup\dots\cup A_k$ such that $W$ is a constant $p_{i,j}$ on each product $A_i\times A_j$, $i,j\in [k]$. We can assume that each $A_i$ has positive measure, since changing the values of $W$ on a null subset of $[0,1]^2$ does not affect the distribution of~$\I G(n,W)$. 

Before stating our LDP, let us point out the difference between the random graphs $\I G((a_1,\dots,a_k),(p_{i,j})_{i,j\in [k]})$ and $\I G(a_1+\dots+a_k,W)$ when $W$ and $(p_{i,j})_{i,j\in [k]}$ are as above.
% in the previous paragraph
In the former model, we have exactly $a_i$ vertices in the $i$-th block for each $i\in [k]$. In the latter model, each vertex is put into one of the $k$ blocks with the probabilities given by the measures of $A_1,\dots,A_k$, independently of the other vertices; thus the number
of vertices in each block is binomially distributed.
%the Lebesgue measure $\lambda(A_i)$ of~$A_i$ (thus the total number of vertices in the $i$-th block has the binomial distribution with parameters $(a_1+\dots+a_k,\lambda(A_i))$).
\hide{
Before stating our LDP, let us point out that some differences and relations between the random graphs $\I G(\V a,\V p)$ and $\I C(n,W)$ when $W$ and $\V p$ are as in the previous paragraph, $\V a=(a_1,\dots,a_k)$, and $n=a_1+\dots+a_k$.
In $\I G(\V a,\V p)$ we take exactly $a_i$ vertices from the $i$-th block. On the other hand, in $\I C(n,W)$ each vertex is put into one the blocks independently of the other vertices with probabilities equal to the measures
$\lambda(A_1),\dots,\lambda(A_k)$ of the parts; thus the total number of vertices in the $i$-th block has the binomial distribution with parameters $(n,\lambda(A_i))$. Once the vertices are assigned to blocks, the rule for generating the edges is same in both models: a pair connecting the $i$-th and $j$-th block is made an edge with probability $p_{i,j}$. Thus if we condition $\I C(n,W)$ on the blocks being consecutive intervals of $[n]$ of lengths respectively $a_1,\dots,a_k$, then the conditional distribution will be exactly the distribution of~$\I G(\V a,\V p)$.}%
It comes as no surprise that if we consider large deviations for $(\tR_{n,W})_{n\in\I N}$ at speed 
$n^2$ then the rate function depends only on $(p_{i,j})_{i,j\in [k]}$ but not on the (non-zero) measures of the parts $A_i$ since, informally speaking,  the price we ``pay'' to get any desired distribution of vertices per parts is multiplicative $\me^{-O(n)}$, which is negligible for speed~$n^2$.

\begin{theorem}\label{th:ourLDP}
	Let $W$ be a $k$-step graphon with $k$ non-null parts whose values are encoded by a symmetric $k\times k$ matrix $\V p\in [0,1]^{k\times k}$. Define
	\beq{eq:R}
	R_{\V p}(\T U):=\inf_{\V\alpha\in [0,1]^k\atop \alpha_1+\ldots+\alpha_k=1} J_{\V\alpha,\V p}(\T U),\quad U\in\C W.
	\eeq
	Then the function $R_{\V p}:\tW\to [0,\infty]$ is lower semi-continuous with respect to the metric $\delta_\Box$. Moreover, the sequence of measures $(\tR_{n,W})_{n\in\I N}$ on $(\tW,\delta_\Box)$ satisfies an LDP with speed $n^2$ and rate function~$R_{\V p}$.
	\end{theorem}

For $k=1$, we recover the LDP result of Chatterjee and Varadhan~\cite{ChatterjeeVaradhan11} (that is, Theorem~\ref{th:CV}). 

Initially, we proved Theorem~\ref{th:ourLDP} independently of the work by Borgs et al~\cite{BCGPS}, by first proving an LDP for what we call \emph{$k$-coloured graphons} (that are defined in Section~\ref{ColGraphons}). Since our original proof of Theorem~\ref{th:ourLDP} is quite long and shares many common steps with the proof from~\cite{BCGPS} (with both being built upon the method of
Chatterjee and Varadhan~\cite{ChatterjeeVaradhan11}), we decided to derive Theorem~\ref{th:ourLDP} from the results in~\cite{BCGPS} with a rather short proof, also strengthening the LDP of Borgs et al~\cite{BCGPS} in the process.

\medskip This paper is organised as follows. In Section~\ref{prelim} we give further definitions (repeating some definitions from the Introduction) and provide some standard or easy results that we will need later. Section~\ref{ColGraphons} introduces $k$-coloured graphons and proves a compactness result. This result is used in Section~\ref{Rate} to prove that the functions $J_{\V a,\V p}$ and $R_{\V p}$ are lower semi-continuous. The large deviation principles stated in Theorems~\ref{th:GenLDP} and~\ref{th:ourLDP} are proved in Section~\ref{GenLDP} and~\ref{ourLDP} respectively.

\section{Preliminaries}\label{prelim}

Recall that the relative entropy $h_p$ was defined in~\eqref{eq:hp} and observe the conventions that $0\log(0)=0\log\left(\frac{0}{0}\right)=0$, $h_0(\rho)=+\infty$ whenever $\rho\not=0$ and $h_1(\rho)=+\infty$ whenever $\rho\not =1$.
The \emph{indicator function} $\I 1_X$ of a set $X$ assumes value $1$ for every $x\in X$ and 0 otherwise.

A measurable space $(\Omega,\C A)$ is called \emph{standard}
if there is a Polish topology on $\Omega$ whose Borel $\sigma$-algebra is equal to~$\C A$. Given a measure
$\mu$ on $(\Omega,\C A)$, we call a subset of $\Omega$ \emph{measurable} if it belongs to the \emph{completion} of $\C A$
by $\mu$, that is, the $\sigma$-algebra generated by $\C A$ and $\mu$-null sets.
We will usually omit $\sigma$-algebras from our notation.

Unless specified otherwise, the interval $[0,1]$ of reals is always equipped with the Lebesgue measure, denoted by~$\lambda$. By $\lambda^{\oplus k}$ we denote the completion of the $k$-th power of $\lambda$, that is, $\lambda^{\oplus k}$ is the Lebesgue measure on~$[0,1]^k$

Denote as $\Ak$ the set of all ordered partitions of $[0,1]$ into $k$ measurable sets.
For a non-zero vector $\V\alpha=(\alpha_1,\dots,\alpha_k)\in [0,\infty)^k$, we let $\Aalpha\subseteq\Ak$ to be the set of all ordered partitions of $[0,1]$ into $k$ measurable sets such that the $i$-th set has Lebesgue measure exactly~$\alpha_i/\|\V\alpha\|_1$. Recall that $(\cI{\V\alpha}{1},\dots,\cI{\V\alpha}{k})\in\Aalpha$ denotes the partition of $[0,1]$ into consecutive intervals whose lengths are given by~$\V\alpha/\|\V\alpha\|_1$  (where each dividing point is assigned  to e.g.\ its right interval for definiteness).

We will also need the following result. 
\hide{ (whose proof can be found in e.g.~\cite[Theorem 3.4.23]{Srivastava98cbs}).

\begin{theorem}[Isomorphism Theorem for Measure Spaces]\label{th:ISMS}
	For every two atomless standard probability spaces, 
	%$(\Omega,\mu)$ and $(\Omega',\mu')$ 
	there is a measure-preserving Borel isomorphism between them.
\end{theorem}	
}%

\begin{theorem}\label{th:ISMS}
	For every two atomless standard measure spaces 
$(\Omega,\mu)$ and $(\Omega',\mu')$, 
%with $0<\mu(\Omega)=\mu(\Omega')<\infty$, 
and Borel subsets $A\subseteq \Omega$ and $A'\subseteq \Omega'$ with $0<\mu(A)=\mu'(A')<\infty$,
% of the same positive measure,
	there is a measure-preserving Borel isomorphism between $A$ and~$A'$.
\end{theorem}

\begin{proof} The case when $A=\Omega$ and $A'=\Omega'$ amounts to the Isomorphism Theorem for Measure Spaces, whose proof can be found in e.g.~\cite[Theorem 3.4.23]{Srivastava98cbs}. The general case follows by restricting everything to $A$ and $A'$, and noting that the obtained measure spaces are standard by e.g.\ \cite[Theorem 3.2.4]{Srivastava98cbs}.\end{proof}

\subsection{Graphons}\label{graphons}

A \emph{graphon} $U$ is a function  $U:[0,1]^2\to[0,1]$ which is \emph{symmetric} (that is,
$U(x,y)=U(y,x)$ for all $x,y\in [0,1]$) and \emph{measurable} (that is, for every $a\in \I R$, the level set $\{W\le a\}$ is a (Lebesgue) measurable subset of $[0,1]^2$). 
Recall that we denote the set of all graphons by~$\mathcal{W}$.
We define the \emph{cut-norm} $d_\Box:\C W^2\to [0,1]$ by 
 \beq{eq:CutNorm}
d_\Box(U,V):=\sup_{A,B\subseteq [0,1]} \left|\int_{A\times B}\left(U-V\right) \dd\lambda^{\oplus 2}\right|,\quad U,V\in \mathcal{W},
\eeq
where the supremum is taken over all pairs of measurable subsets of $[0,1]$.
For a function $\phi:[0,1]\to [0,1]$ and a graphon $U$, 
the \emph{pull-back} $U^\phi$ of $U$ along $\phi$ is defined 
by 
$$
 U^{\phi}(x,y):=U(\phi(x),\phi(y)),\quad x,y\in [0,1].
 $$
The \emph{cut-distance} $\delta_\Box:\C W^2\to [0,1]$ can be defined as
\beq{eq:CutDistance}
\delta_{\Box}(U,V):=\inf_{\phi,\psi} d_\Box (U^\phi,V^\psi),
\eeq
where the infimum is taken over all measure-preserving maps $\phi,\psi:[0,1]\to [0,1]$ (then the pull-backs $U^\phi$ and $V^\psi$ are necessarily measurable functions).
See \Lo{Section~8.2} for more details and, in particular, \Lo{Theorem 8.13} for some alternative definitions that give the same distance. It can be easily verified that $\delta_\Box$ is a pseudo-metric on~$\C W$.
Recall that we denote the metric quotient by $\tW$ and  the equivalence class of $U\in \mathcal{W}$ by~$\widetilde{U}$. For $U\in\C W$ and $\eta\ge 0$, we will denote the closed radius-$\eta$ ball around $\T U$ in 
%$\C W$ and 
$\tW$ by
$$
	\Sball(\widetilde U,\eta):= \left\{\T V\in \T{\C W}\mid \delta_\Box(U,V)\le\eta\right\}.
$$

\begin{remark} Without affecting $(\tW,\delta_\Box)$, we could have defined a graphon as a Borel symmetric function $[0,1]^2\to [0,1]$ and required that the measure-preserving maps in the definition of $\delta_\Box$ are Borel. Then some parts could be simplified (for example, the first claim of Lemma~\ref{lm:AlmostMP} would not be necessary as the function $U^\phi$ would be Borel for every Borel~$\phi$). However, 
		we prefer to use the (now standard) conventions from Lov\'asz' book~\cite{Lovasz:lngl}.			
\end{remark}

%We will need the following auxiliary results. One is the result of Lov\'asz and Szegedy~\cite[Theorem~5.1]{LovaszSzegedy07gafa} that the metric space $(\tW,\delta_\Box)$ is compact. Another one is as follows.
	
We will need the following auxiliary result.

\begin{lemma}\label{lm:AlmostMP}
	Let $U$ be a graphon and $\phi:[0,1]\to [0,1]$ be a measurable function such that the \emph{push-forward measure} $\phi_*\lambda$ (defined by $(\phi_*\lambda)(X):=\lambda(\phi^{-1}(X))$ for measurable $X\subseteq [0,1]$) satisfies $\phi_*\lambda\ll \lambda$, that is, 
 is absolutely continuous with respect to the Lebesgue measure~$\lambda$. Then  $U^\phi$ is a graphon. Moreover, 
	if the Radon-Nikodym derivative $D:=\frac{\dd (\phi_*\lambda)}{\dd \lambda}$ satisfies $D(x)\le 1+\e$ for a.e.\ 
	$x\in[0,1]$ then
	\beq{eq:AlmostMP}
	\delta_\Box(U,U^\phi)\le 2\e.
	\eeq
\end{lemma}

\begin{proof} 
		The function $U^\phi:[0,1]^2\to [0,1]$ is clearly symmetric so we have to show that it is measurable.
	Since the pre-image under $\phi$ of any $\lambda$-null set is again $\lambda$-null by our assumption $\phi_*\lambda\ll \lambda$, there is a Borel map  $\psi:[0,1]\to[0,1]$  such that the set 
	$$X:=\{x\in [0,1]:\psi(x)\not=\phi(x)\}$$ 
	is $\lambda$-null. (For a proof, see e.g.~\cite[Proposition~2.2.5]{Cohn13mt}.)
Take any $\rho\in\I R$. The set 
$$
 A:=\{(x,y)\in [0,1]^2\mid U(x,y)\le \rho\}
 $$ 
 is measurable so by e.g.~\cite[Proposition~1.5.2]{Cohn13mt} there are $B,N\subseteq [0,1]^2$ such that $B$ is Borel, $N$ is $\lambda^{\oplus 2}$-null and $A\bigtriangleup B\subseteq N$, where $A\bigtriangleup B:=(A\setminus B)\cup (B\setminus A)$ denotes the \emph{symmetric difference} of the sets $A$ and $B$. The pre-image of $N$ under the Borel map $\psi^{\oplus2}(x,y):=(\psi(x),\psi(y))$ is also $\lambda^{\oplus 2}$-null for otherwise
 this would contradict the absolute continuity $(\phi_*\lambda)^{\oplus 2}\ll \lambda^{\oplus 2}$ (which follows from
$\phi_*\lambda\ll \lambda$ by the Fubini-Tonelli Theorem for Complete Measures).
 Thus the level set 
 $\{U^\phi\le \rho\}$
 %$\{(x,y)\in [0,1]^2\mid U^\psi(x,y)\le \rho\}$ 
 is Lebesgue measurable since its symmetric difference with the Borel set $(\psi^{\oplus2})^{-1}(B)$ is a subset of the null set $(\psi^{\oplus2})^{-1}(N)\cup (X\times [0,1])\cup ([0,1]\times X)$. As $\rho\in\I R$ was arbitrary, $U^\phi$ is a measurable function and thus a graphon.

For the second part, it will be convenient to use the following generalisation of a graphon (which will not used anywhere else in this paper except this proof).
% that allows us to replace $U$ (resp.\ $U^\phi$) by a weakly isomorphic graphon where, informally speaking, each vertex $x\in [0,1]$ is replaced by the interval $[0,1]$ (resp.\ $[0,f(x)]$) of ``twins'' of~$x$.  
Namely, by
a \emph{generalised graphon} we mean a triple $(V,\Omega,\mu)$ where $(\Omega,\mu)$ is an atomless standard probability space  and $V:(\Omega^2,\mu^{\oplus 2})\to [0,1]$ is a symmetric measurable function. In the special case $(\Omega,\mu)=([0,1],\lambda)$ we get our notion of a graphon from the Introduction. Most definitions
and results extend with obvious modifications from graphons to generalised graphons  (see \Lo{Chapter~13.1} for details). %By the Isomorphism Theorem for Measure Spaces (see e.g.~\cite[Theorem 3.4.23]{Srivastava98cbs}), 
%Each generalised graphon is weakly isomorphic to a graphon (on $([0,1],\lambda)$), so we do not enlarge $\tW$ by adding generalised graphons.
In particular, we will need the facts that if $(V,\Omega,\mu)$ is a generalised graphon and
$\phi: (\Omega',\mu')\to(\Omega,\mu)$ is a measure-preserving  map between standard probability spaces,
then the function $V^\phi$ is measurable (which can be proved by adapting the proof of the first part of the lemma) and
% the appropriately defined	cut-distance $\delta_\Box$ satisfies
\beq{eq:51}
\delta_\Box((V,\Omega,\mu),(V^\phi, \Omega',\mu'))=0,
\eeq
where we define $V^\phi(x,y):=V(\phi(x),\phi(y))$ for $x,y\in \Omega'$ and $\delta_\Box$ is the extension of the cut-distance to generalised graphons via the obvious analogues of~\eqref{eq:CutNorm} and~\eqref{eq:CutDistance}.

Let us return to the proof of the second part of the lemma.
	We can assume that the set $\{x\in [0,1]\mid D(x)\not=1\}$ has positive measure for otherwise $\phi$ is a measure-preserving map and $\delta_\Box(U,U^\phi)=0$ by~\eqref{eq:51}, as required. 
	Let $(V,[0,1]^2,\lambda^{\oplus 2})$ be the generalised graphon defined by $V((x,y),(x',y')):=U(x,x')$, for $x,y,x',y'\in [0,1]$. Thus $V=U^\pi$, where
	$\pi:[0,1]^2\to [0,1]$ is the (measure-preserving) projection on the first coordinate. 
	By~\eqref{eq:51}, it holds that
	$$
	\delta_\Box ((V,[0,1]^2,\lambda^{\oplus 2}),U)=0.
	$$
	
	By changing $D$ on a $\lambda$-null set, we can make it a Borel function with $D(x)\le 1+\e$ for every $x\in [0,1]$. Then
	$$
	\Omega:=\{(x,y)\in [0,1]\times \I R\mid 0\le y\le D(x)\}
	$$ 
	is a Borel subset of $[0,1]\times \I R$ (see e.g.\ \cite[Example 5.3.1]{Cohn13mt}) and thus induces a standard measurable space. Let $\mu$ be the restriction of the Lebesgue measure on $\I R^2$ to~$\Omega$. Define $W:\Omega^2\to [0,1]$ by $W((x,y),(x',y')):=U(x,x')$ for $(x,y),(x',y')\in \Omega$. Thus $U^\phi$ and $(W,\Omega,\mu)$ are measure-preserving pull-backs of the generalised graphon $(U,[0,1],\phi_*\mu)$ 
along respectively the map $\phi$ and the projection $\Omega\to[0,1]$ on the first coordinate.
	 Therefore we have by~\eqref{eq:51} and the Triangle Inequality for $\delta_\Box$ that
	$$
\delta_\Box(U^\phi,(W,\Omega,\mu))\le \delta_\Box(U^\phi,(U,[0,1],\phi_*\mu))+\delta_\Box((U,[0,1],\phi_*\mu),(W,\Omega,\mu))=0.
$$  
	
	Thus it suffices to show that the cut-distance $\delta_\Box$ between $(V,[0,1]^2,\lambda^{\oplus 2})$ and $(W,\Omega,\mu)$ is at most~$2\e$.
	The functions $V$ and $W$ and the measures $\lambda^{\oplus 2}$ and $\mu$
	coincide on~$X^2$, where $X:=[0,1]^2\cap \Omega$. 
	Since $D\not= 1$ on a set of positive measure, it holds that $\mu(X)<1$. 
The Borel subsets $[0,1]^2\setminus X$ and $\Omega\setminus X$ of $\I R^2$ have the same positive Lebesgue measure and thus,
	%by the Isomorphism Theorem for Measure Spaces (Theorem~\ref{th:ISMS}), 
	by Theorem~\ref{th:ISMS}, there is a Borel measure-preserving bijection $\psi$ between them. By letting $\psi$ be the identity function on $X$, we get a Borel measure-preserving bijection $\psi:\Omega\to [0,1]^2$. 
	Since $D\le 1+\e$, we have that $\Omega\setminus X\subseteq [0,1]\times [1,1+\e]$ has measure at most~$\e$. 	
The function $W$ and the pull-back $V^\psi$, as maps $\Omega^2\to [0,1]$, coincide on the set $X^2$ of measure at least $(1-\e)^2\ge 1-2\e$. It follows that the $d_\Box$-distance between them at most $2\e$. (Indeed, when we compute it via the analogue of~\eqref{eq:CutNorm}, the integrand is bounded by 1 in absolute value and is non-zero on a set of measure at most~$2\e$.) This finishes the proof of the lemma.\end{proof}

Informally speaking, the following result states that if we delete a small subset of $[0,1]$ and stretch the rest of a graphon uniformly then the new graphon is close to the original one.

\begin{lemma}\label{lm:Delete01}
	Let $U\in \C W$ be a graphon, $s\in (0,1]$ be a non-zero real and $\phi:[0,1]\to [0,1]$ be the map that sends $x$ to~$sx$. 
	Then $U^\phi$ is a graphon and $\delta_\Box(U,U^\phi)\le 2(\frac1s -1)$.
\end{lemma}
\begin{proof}  Clearly, the push-forward $\phi_*\lambda$ is the uniform probability measure on $[0,s]$ so the Radon-Nikodym derivative $\frac{\dd( \phi_*\lambda)}{\dd\lambda}$ is a.e.\ $1/s$ on~$[0,s]$ and $0$ on~$[s,1]$. The result now follows from Lemma~\ref{lm:AlmostMP}. 
\end{proof}

\subsection{Large deviations for compact metric spaces}\label{LDP}

Recall that the definition of a large deviation principle was given in Definition~\ref{df:LDP}. Since we will be dealing with LDPs for compact metric spaces only, we may use the following alternative characterisation that follows from \DZ{Theorems 4.1.11 and 4.1.18} (see also~\RS{Exercise~2.24}). 

\begin{lemma}\label{lm:LDP}
	Let $(X,d)$ be  a compact metric space, $s:\I N\to (0,\infty)$ satisfy $s(n)\to\infty$, and $I:X\to [0,\infty]$ 
	be a lower semi-continuous function on~$(X,d)$. Then
	a sequence of Borel probability measures $(\mu_n)_{n\in \I N}$ on $(X,d)$ satisfies an LDP 
	with speed $s$
	and rate function $I$
	if and only 
	\begin{eqnarray}
		\lim_{\eta\to 0}\liminf_{n\to\infty} \frac1{s(n)}\,{\log\Big(\mu_n\big(
			\{y\in X\mid d(x,y)\le \eta\}
			\big)\Big)} 
		&\ge& -I(x),\quad\mbox{for every $x\in X$,}\label{eq:lower}\\
		\lim_{\eta\to 0}\limsup_{n\to\infty} \frac1{s(n)}\,{ \log\Big(\mu_n\big(
			\{y\in X\mid d(x,y)\le \eta\}
			\big)\Big)} &\le & -I(x),\quad\mbox{for every $x\in X$}.\label{eq:upper}
	\end{eqnarray}
\end{lemma}
	
	In fact, under the assumptions of Lemma~\ref{lm:LDP}, the bounds~\eqref{eq:lowerGen} and~\eqref{eq:lower} (resp.\ \eqref{eq:upperGen} and~\eqref{eq:upper}) are equivalent to each other. So we will also refer to 
	\eqref{eq:lower} and~\eqref{eq:upper}
	as the \emph{lower bound} and the \emph{upper bound} respectively.

\section{Coloured graphons}\label{ColGraphons}

The definitions and results of this section are needed in order to establish the lower semi-continuity of the functions $J_{\V \alpha,\V p}$ and~$R_{\V p}$.

Fix $k\in \mathbb{N}$.
By a \emph{$k$-coloured graphon} we mean a pair $(W,\mathcal{A})$ where $W\in \mathcal{W}$ and $\mathcal{A}\in \Ak$. (One can view the partition $\mathcal{A}$ as a $k$-colouring of $[0,1]$.)
Write $\Wk:=\C W\times \Ak$ for the space of all $k$-coloured graphons.
We define the pseudo-metric $d^{(k)}_\Box$ (the analogue of the cut norm $d_\Box$) on $\Wk$ as
\begin{equation*}
d^{(k)}_\Box((U,\mathcal{A}),(V,\mathcal{B})):=  
\sup_{C,D\subseteq [0,1]} \sum_{i,j\in [k]}\left| \int_{C\times D} ({\I 1}_{A_i\times A_j}U-{\I 1}_{B_i\times B_j}V) \dd\lambda^{\oplus 2}\right|
 + \sum_{i\in [k]}\lambda(A_i\triangle B_i),
\end{equation*}
 for $(U,\mathcal{A}),(V,\mathcal{B})\in \Wk$, where $\C A=(A_1,\dots,A_k)$ and $\C B=(B_1,\dots,B_k)$.
 Informally speaking, two $k$-coloured graphons are close to each other in $d^{(k)}$ if
 they have similar distributions of coloured edges across cuts, where an edge is coloured by the colours of its endpoints. The second term is added so that e.g.\ we can distinguish two constant-0 graphons with different part measures.
 
The \emph{cut distance} for coloured graphons is then defined as
\beq{eq:DeltaK}
\delta^{(k)}_\Box((U,\mathcal{A}),(V,\mathcal{B})):=\inf_{\phi,\psi} d^{(k)}_\Box((U,\mathcal{A})^\phi,(V,\mathcal{B})^\psi),
\eeq
where the infimum is taken over measure-preserving maps $\phi,\psi:[0,1]\to [0,1]$ and we denote
$(U,\mathcal{A})^\phi:=(U^\phi,\C A^\phi)$ and $\C A^\phi:=(\phi^{-1}(A_1),\dots,\phi^{-1}(A_k))$ 
with $(A_1,\dots,A_k)$ being the parts of~$\C A$.
As in the graphon case (compare with e.g.\ \Lo{Theorem 8.13}), some other definitions give the same distance (e.g.\ it is enough to take the identity function for $\psi$). We chose this definition as it is immediately clear from it that the function $\delta^{(k)}_\Box$ is symmetric and defines a pseudo-metric.
We write $\tWk$ for the corresponding quotient, where we identify two $k$-coloured graphons 
at $\delta^{(k)}_\Box$-distance~0.

\begin{theorem}\label{th:compact colorod graphon}
The metric space $(\tWk,\delta^{(k)}_\Box)$ is compact. 
\end{theorem}
\begin{proof}
The proof is obtained by the obvious adaptation of the proof of Lov\'asz and Szegedy~\cite[Theorem~5.1]{LovaszSzegedy07gafa} (see also~\cite[Theorem 9.23]{Lovasz:lngl}) that the space $(\tW,\delta_\Box)$ is compact.

Let $(W_n,\mathcal A_n)_{n\in\mathbb N}$ be an arbitrary sequence of elements of~$\Wk$. We have to find a subsequence 
that converges to some element of $\Wk$ with respect to~$\delta_\Box^{(k)}$. 

When dealing with the elements of $\Wk$,
%measurable partitions, 
we can ignore null subsets of $[0,1]$; thus all relevant statements, e.g.\ that one partition refines another, are meant to hold almost everywhere.

For $n\in\mathbb N$, let the parts of $\mathcal A_n$ be $(A_{n,1},\dots,A_{n,k})$ and, by applying a measure-preserving bijection to $(W_n,\mathcal A_n)$, assume by Theorem~\ref{th:ISMS} that the colour classes $A_{n,1},\dots,A_{n,k}\subseteq [0,1]$ are all intervals, 
%(where a dividing point can be assigned to either of the two intervals), 
coming in this order. By passing to a subsequence, assume that, for each $i\in [k]$, the length of $A_{n,i}$ converges to some $\alpha_i\in [0,1]$ as $n\to\infty$. With $\V\alpha:=(\alpha_1,\dots,\alpha_k)$, this gives rise to the ``limiting'' partition 
$$
\mathcal A
:=(\cI{\V\alpha}{1},\dots,\cI{\V\alpha}{k})
\in\Aalpha
$$ of $[0,1]$ into intervals.

Let $m_1:=k$ and inductively for $\ell=2,3,\ldots$ let $m_\ell$ be sufficiently large such that for every graphon $W$ and a measurable partition $\mathcal A'$ of $[0,1]$ with $|\mathcal A'|\le m_{\ell-1}$ there is a measurable partition $\mathcal P=(P_1,\dots,P_m)$ of $[0,1]$ refining $\mathcal A'$ such that $m\le m_\ell$ and $d_{\Box}(W,W_{\mathcal P})\le 1/\ell$. Here $W_{\mathcal P}$ denotes the projection of $W$ to the space of $\C P$-step graphons; namely, for every $i,j\in [m]$ with $P_i\times P_j$ non-null in $\lambda^{\oplus 2}$, $W_{\C P}$ assumes the constant value $\frac1{\lambda(P_i)\lambda(P_j)}\int_{P_i\times P_j} W\dd\lambda^{\oplus 2}$ on $P_i\times P_j$ (and, say, $W_{\C P}$ is defined to be 0 on all $\lambda^{\oplus 2}$-null products $P_i\times P_j$).
Such a number $m_\ell$
exists by~\cite[Lemma 9.15]{Lovasz:lngl}, a version of the Weak Regularity Lemma for graphons.
 
For each $n\in\mathbb N$, we do the following. Let $\mathcal P_{n,1}:=\mathcal A_n$ and, inductively on $\ell=2,3,\dots$, let $\mathcal P_{n,\ell}$ be the partition with at most $m_{\ell}$ parts obtained by applying the above Weak Regularity Lemma to $(W_n,\mathcal P_{n,\ell-1})$. By adding empty parts to $\mathcal P_{n,\ell}$, for each $\ell\ge 1$, we can assume that it has the same number of parts (namely, $m_\ell$) of each colour, that is, we can denote its parts as $(P_{n,\ell,i,j})_{i\in [k], j\in [m_\ell]}$, so that $P_{n,\ell,i,j}\subseteq A_{n,i}$ for all $(i,j)\in [k]\times [m_\ell]$.
Also, define $W_{n,\ell}:=(W_n)_{\mathcal P_{n,\ell}}$ to be the projection of the graphon $W_n$ on the space of $\mathcal P_{n,\ell}$-step graphons.

Then, iteratively for $\ell=2,3,\dots$, repeat the following. Find a measure-preserving bijection $\phi:[0,1]\to[0,1]$ such that $(\mathcal P_{n,\ell})^\phi$ is a partition into intervals and $\phi$ preserves the previous partitions
$\mathcal P_{n,1},\dots,\mathcal P_{n,\ell-1}$  (each of which is a partition into intervals by induction on~$\ell$). Then, for each $m\ge \ell$, replace $(W_{n,m},\mathcal P_{n,m})$ by $(W_{n,m},\mathcal P_{n,m})^\phi$. When we are done with this step, the following properties hold for each integer~$\ell\ge 2$:
\begin{itemize}
	\item\label{it:P1} $\delta_\Box^{(k)}((W_{n,\ell},\mathcal A_n),(W_n,\mathcal A_n))\le 1/\ell$;
	\item\label{it:P3} %If $\ell\ge 2$ then 
The partition $\mathcal P_{n,\ell}$ refines $\mathcal P_{n,\ell-1}$ (and, inductively, also refines $\mathcal A_{n}=\mathcal P_{n,1}$);
	\item\label{it:P2} $|\mathcal P_{n,\ell}|= k m_\ell$ with exactly $m_\ell$ parts assigned to each colour class of $\mathcal A_n$.
\end{itemize}

Next, iteratively for $\ell=1,2,\dots$, we pass to a subsequence of $n$ so that for every $(i,j)\in [k]\times [m_\ell]$, the length of the interval $P_{n,\ell,i,j}$ converges, and for every pair $(i,j),(i',j')\in [k]\times [m_\ell]$, the common value of the step-graphon $W_{n,\ell}$ on $P_{n,\ell,i,j}\times P_{n,\ell,i',j'}$ converges.
% (where, for definiteness, the value is defined to be e.g.\ 0 if the product is empty). 
	It follows that the sequence $W_{n,\ell}$ converges pointwise to some graphon $U_\ell$ which is itself a step-function with $km_\ell$ parts that are intervals. We use diagonalisation to find a subsequence of $n$ so that, for each $\ell\in\mathbb N$, $W_{n,\ell}$ converges to some step-graphon $U_\ell$ a.e.\ as $n\to\infty$, with the step partition $\mathcal P_\ell$ of  $U_\ell$  consisting of $km_\ell$ intervals and refining the partition~$\C A$. 
\hide{Moreover, 	we can denote the parts of $\mathcal P_\ell$ as
	$(P_{\ell,i,j})_{i\in[k], m\in[m_\ell]}$, so that $P_{\ell,i,1},\dots,P_{\ell,i,m_\ell}$ form a partition of~$\cI{\V\alpha}{i}$ for each $i\in [k]$.
}

It follows that, for all $s<t$ in $\I N$, the partition $\mathcal P_{t}$ is a refinement of $\mathcal P_s$ and, moreover,  $U_s$ is the conditional expectation $\mathbb E[U_t|\mathcal P_s]$ a.e. 
As observed in the proof of Lov\'asz and Szegedy~\cite[Theorem~5.1]{LovaszSzegedy07gafa},
this (and $0\le U_t\le 1$ a.e.) implies by the Martingale Convergence Theorem that $U_\ell$ converge a.e.\ to some graphon $U$ as $\ell\to\infty$. By the Dominated Convergence Theorem, $U_\ell\to U$ also in the $L^1$-distance.

We claim that $\delta_\Box((W_n,\mathcal A_n),(U,\mathcal A))\to 0$ as $n\to\infty$ (after we passed to the subsequence defined as above). Take any $\e>0$. Fix an integer $\ell>4/\e$ such that $\|U-U_\ell\|_1\le \e/4$. Given $\ell$, fix $n_0$ such that  for all $n\ge n_0$ we have $\|U_\ell-W_{n,\ell}\|_1\le \e/4$ and $\sum_{i=1}^k \lambda(\cI{\V\alpha}{i}\bigtriangleup A_{n,i})\le \e/4$. Then, for every $n\ge n_0$ we have
\begin{eqnarray*}
\delta_\Box^{(k)}((U,\mathcal A),(W_n,\mathcal A_n))&\le &
d_\Box^{(k)}((U,\mathcal A),(U_\ell,\mathcal A))
+ d_\Box^{(k)}((U_\ell,\mathcal A),(W_{n,\ell},\mathcal A_n))\\
&+& \delta_\Box^{(k)}((W_{n,\ell},\mathcal A_n),(W_n,\mathcal A_n))\\
&\le& \|U-U_\ell\|_1 + \|U_\ell-W_{n,\ell}\|_1 + \sum_{i=1}^k \lambda(\cI{\V\alpha}{i}\bigtriangleup A_{n,i}) + 1/\ell\\
&\le& \e/4+\e/4+\e/4+\e/4\ =\ \e.
\end{eqnarray*}
 Since $\e>0$ was arbitrary, the claim is proved. Thus the metric space $(\tWk,\delta^{(k)}_\Box)$ is indeed compact.
\end{proof}

\section{The lower semi-continuity of $J_{\V\alpha,\V p}$ and $R_{\V p}$}\label{Rate}

For this section we fix an integer $k\ge 1$, a symmetric $k\times k$ matrix $\V p=(p_{i,j})_{i,j\in [k]}\in [0,1]^{k\times k}$ and a non-zero real vector $\V\alpha\in [0,\infty)^k\NoZero$.
We show that the functions $J_{\V\alpha,\V p}$ and  $R_{\V p}$
are lower semi-continuous functions from $(\tW,\delta_\Box)$ to $[0,+\infty]$.

Let $\Gamma:\Wk\to \C W$ be the map that forgets the colouring, i.e., $\Gamma(U,\mathcal{A}):=U$
for $(U,\C A)\in \Wk$.
For $i,j\in [k]$, let the map $\Gamma_{i,j}:\Wk\to \C W$ send $(U,(A_1,\dots,A_k))\in \Wk$  to the graphon $V$ defined as
$$
%(\Gamma_{i,j}(U,(A_1,\dots,A_k)))
V(x,y):=\left\{\begin{array}{ll} 
	U(x,y),& (x,y)\in (A_i\times A_j)\cup (A_j\times A_i),\\
	p_{i,j},& \mbox{otherwise,}
\end{array}
\right.\quad \mbox{for $x,y\in [0,1]$.}
$$

\begin{lemma}\label{lm:Lipschitz}
	The maps $\Gamma$ and $\Gamma_{i,j}$, for ${i,j\in [k]}$, are $1$-Lipschitz maps from $(\Wk,d_\Box^{(k)})$ to $(\C W,d_\Box)$.
\end{lemma}
\begin{proof}
	First, consider $\Gamma:\Wk\to \mathcal{W}$. Take arbitrary $(U,\mathcal{A}),(V,\mathcal{B})\in \Wk$.
	Let $\C A=(A_1,\dots,A_k)$ and $\mathcal{B}=(B_1,\dots,B_k)$. Clearly, the pairwise products $A_i\times A_j$ (resp.\ $B_i\times B_j$) for $i,j\in [k]$ partition $[0,1]^2$. Thus
	we have
	\begin{eqnarray}		
			d_\Box(\Gamma(U,\mathcal{A}),\Gamma(V,\mathcal{B})) &= & d_\Box(U,V) \
			= \ \sup_{C,D\subseteq [0,1]} \left |\int_{C\times D} (U-V) \dd\lambda^{\oplus 2}\right |\nonumber \\
			&\le & \sup_{C,D\subseteq [0,1]}\sum_{i,j\in [k]}\left|\int_{C\times D} (U\,{\I 1}_{A_i\times A_j}-V\,\I 1_{B_i\times B_j}) \dd\lambda^{\oplus 2}\right| \nonumber\\
			&\le &  d_{\Box}^{(k)}((U,\mathcal{A}),(V,\mathcal{B})).\label{eq:Ga}
			\end{eqnarray}
	Thus the function $\Gamma$ is indeed $1$-Lipschitz.

	The claim about $\Gamma_{i,j}$ follows by observing that
	$$d_\Box(\Gamma_{i,j}(U,\mathcal{A}),\Gamma_{i,j}(V,\mathcal{B}))\le d_\Box(\Gamma(U,\mathcal{A}),\Gamma(V,\mathcal{B}))$$
	for every $(U,\mathcal{A}),(V,\mathcal{B})\in\Wk$.\end{proof}

\begin{lemma}\label{lm:Cont}
	Let $F$ be $\Gamma$ or $\Gamma_{i,j}$ for some $i,j\in [k]$.
	Then $F$  gives rise to a well-defined function $\tWk\to\tW$ which, moreover, is $1$-Lipschitz
	as a function from $(\tWk,\delta_\Box^{(k)})$ to~$(\tW,\delta_\Box)$.
\end{lemma}
\begin{proof} 
	
	Take any $(U,\C A),(V,\C B)\in\Wk$. Let $\e>0$ be arbitrary. Fix measure-preserving maps $\phi,\psi:[0,1]\to [0,1]$ with $d_\Box^{(k)}((U,\C A)^\psi,(V,\C B)^\phi)<\delta_\Box^{(k)}((U,\C A),(V,\C B))+\e$. By Lemma~\ref{lm:Lipschitz}, 
we have 
\begin{eqnarray*}
	\delta_\Box\left(F(U,\C A),F(V,\C B)\right)&\le& 
	d_\Box\left((F(U,\C A))^\psi,(F(V,\C B))^\phi\right)\ =\
	d_\Box\left(F(U^\psi,\C A^\psi),F(V^\phi,\C B^\phi)\right)\\
	&\le& d_\Box^{(k)}((U^\psi,\C A^\psi),(V^\phi,\C B^\phi))
	%	d_\Box\left(V^\phi,U^\psi\right)
	\ <\ \delta_\Box^{(k)}\left((U,\C A),(V,\C B)\right)+\e.
\end{eqnarray*}
This implies both claims about $F$ as $\e>0$ was arbitrary.
\end{proof}

For $(U,(A_1,\dots,A_k))\in \Wk$, define
\beq{eq:DefIkp}
I^{(k)}_{\V p}(U,(A_1,\dots,A_k)):=\frac12\sum_{i,j\in [k]} \int_{A_i\times A_j} I_{p_{i,j}}(U) \dd\lambda^{\oplus 2}.
\eeq

	In the special case of~\eqref{eq:DefIkp}  when $k=1$ and $p_{1,1}=p$ (and we ignore the second component since ${\bf A}^{(1)}$ consists of just the trivial partition of $[0,1]$ into one part), we get the function
$I_p:\C W\to[0,\infty]$ of Chatterjee and Varadhan defined in~\eqref{eq:IpCV}.

\begin{lemma}\label{lm:lsc}
	The function $I^{(k)}_{\V p}$ gives a well-defined function $\tWk\to [0,+\infty]$ which, moreover, is lower semi-continuous as a function on $(\tWk,\delta^{(k)}_\Box)$.
\end{lemma}
\begin{proof}
Note that we can write
	\beq{eq:IkWCV}
	I^{(k)}_{\V p}(U,\mathcal{A})=\sum_{1\le i\le j\le k} I_{p_{i,j}}(\Gamma_{i,j}(U,\mathcal{A})),\quad (U,\C A)\in\Wk,
	\eeq
	because $\Gamma_{i,j}(U,\C A)$ assumes value $p_{i,j}$ outside of $(A_i\times A_j)\cup (A_j\times A_i)$ while $I_p(p)=0$ for any $p\in [0,1]$. Recall that, by Theorem~\ref{th:CV},
	% of Chatterjee and Varadhan, 
	$I_p$ gives a well-defined function $\tW\to[0,\infty]$ for every $p\in [0,1]$. 
	Thus, by Lemma~\ref{lm:Cont}, the right-hand side of~\eqref{eq:IkWCV} does not change if we replace $(U,\C A)$ by any other element of $\Wk$ at $\delta_\Box^{(k)}$-distance 0. We conclude that $I_{\V p}^{(k)}$ gives a well-defined function on~$\tWk$.
	
	Each composition $I_p\circ \Gamma_{i,j}$ 
	%of the continuous function $\Gamma_{i,j}$ and the lsc function $I_p$ 
	is lsc as a function $(\tWk,\delta_{\Box}^{(k)})\to [0,\infty]$	
	because, for every $\rho\in\I R$, the level set $\{I_p\circ \Gamma_{ij}\le \rho\}$ is closed as the pre-image under the continuous function $\Gamma_{i,j}:\tWk\to\tW$ of the closed 
	set~$\{I_p\le \rho\}$.
	%set~$\{\T U\in \tW\mid I_p(\T U)\le a\}$.
	(Recall that the function $I_p:\tW\to [0,\infty]$ is lsc by Theorem~\ref{th:CV}.)
	Thus $I^{(k)}_{\V p}:\tWk\to[0,\infty]$ is lsc by~\eqref{eq:IkWCV}, as a finite sum of lsc functions.
\end{proof}

Now we are ready to show that $J_{\V\alpha,\V p}$ and $R_{\V p}$ are lsc (in particular, proving Theorem~\ref{th:JLSC}). The argument showing the lower semi-continuity of these functions is motivated by the Contraction Principle (see e.g.\ \DZ{Theorem 4.2.1} or \RS{Section~3.1}). 

\begin{corollary}\label{cr:lsc}
	For every symmetric matrix $\V p\in [0,1]^{k\times k}$ and every non-zero real vector $\V\alpha\in [0,\infty)^k$, the functions $J_{\V\alpha,\V p}$ and $R_{\V p}$ are lower semi-continuous on $(\tW,\delta_\Box)$.	
\end{corollary}

\begin{proof}
	Note that $J_{\V\alpha,\V p}(\T U)$ for any $U\in\C W$ is equal to the infimum of $I_{\V p}^{(k)}(V,\C A)$ over all $(V,\C A)\in \Walpha$ such that $\Gamma(V,\C A)=V$ belongs to~$\T U$. Indeed, for any partition $\C A\in \Aalpha$ of $[0,1]$ one can find by 
	Theorem~\ref{th:ISMS}
	%the Isomorphism Theorem for Measure Spaces (Theorem~\ref{th:ISMS}) 
	a measure-preserving Borel bijection	$\phi$ of $[0,1]$ such that $\C A^\phi$ is equal a.e.\ to $(\cI{\V\alpha}{1},\dots,\cI{\V\alpha}{k})$.% if we ignore sets of measure 0 in both partitions. 

	In the rest of the proof, let us view $\Gamma$ and $I^{(k)}_{\V p}$ as functions on $\tWk$ (by Lemmas~\ref{lm:Cont} and~\ref{lm:lsc}). Thus we have
	\begin{equation}\label{eq:Jap2}
	J_{\V\alpha,\V p}(\T U)=\inf_{\Gamma^{-1}(\T U)\cap\tWalpha}\, I_{\V p}^{(k)},
\quad \mbox{for each $ U\in\C W$},
	\end{equation}
 where $\tWalpha\subseteq \tWk$ denotes the set of all $\delta_\Box^{(k)}$-equivalence classes that intersect $\Walpha$ (equivalently, lie entirely inside $\Walpha$).
 
Take any graphon $U\in\C W$. The pre-image $\Gamma^{-1}(\T U)$ is a closed subset of $\tWk$ by the continuity of $\Gamma$ (Lemma~\ref{lm:Cont}). Also, $\tWalpha$
	is a closed subset of $\tWk$: if 
	$(V,(B_1,\dots,B_k))\in \Wk$ is not in $\Walpha$, 
	%the $\delta_\Box^{(k)}$-equivalence class of $(V,(B_1,\dots,B_k))\in \Wk$ is not in $\tWalpha$, 
	then the $\delta_{\Box}^{(k)}$-ball of radius e.g.\ 
	$$
	 \frac12\,\left\|\,(\lambda(B_i))_{i\in [k]} - \frac1{\|\V\alpha\|_1} \V\alpha\,\right\|_1>0
	 $$
	  around it is disjoint from~$\Walpha$. Recall that the space $\tWk$ is compact by Theorem~\ref{th:compact colorod graphon}. Thus the infimum in~\eqref{eq:Jap2} is taken over a (non-empty) compact set.  As any lsc function attains its infimum on any non-empty compact set and
	$I^{(k)}_{\V p}:\tWk\to [0,+\infty]$ is lsc by Lemma~\ref{lm:lsc}, 
	%for each $U\in \C W$ 
	there is $(V,\C A)\in \Walpha$ such that $V\in\T U$ and $I_{\V p}^{(k)}(\Tk{V,\C A})=J_{\V\alpha,\V p}(\T U)$,
	where $\Tk{V,\C A}$ denotes the $\delta_\Box^{(k)}$-equivalence class of $(V,\C A)$.
	
	Thus for any $\rho\in\I R$ the level set 
	$\{J_{\V\alpha,\V p}\le \rho\}$ 
	is equal to the image of $\{I_{\V p}^{(k)}\le \rho\}\cap\tWalpha$ under~$\Gamma$. Since the function $I_{\V p}^{(k)}$ is lsc by Lemma~\ref{lm:lsc}, the level set $\{I_{\V p}^{(k)}\le \rho\}$ is a closed and thus compact subset of $\tWk$. Thus the set $\{I_{\V p}^{(k)}\le \rho\}\cap\tWalpha$ is compact. Its image $\{J_{\V\alpha,\V p}\le \rho\}$  under the continuous map $\Gamma:\tWk\to \tW$ is compact and thus closed. Since $\rho\in\I R$ was arbitrary, the function $J_{\V\alpha,\V p}:\tW\to [0,\infty]$ is lsc.

	Since $R_{\V p}(\T U)$ is equal to the infimum of $I_{\V p}^{(k)}$ over $\Gamma^{-1}(\T U)$, the same argument (except we do not need to intersect $\Gamma^{-1}(\T U)$ with $\tWalpha$ anywhere) also works for $R_{\V p}$. 
\end{proof}

\section{Proof of Theorem~\ref{th:GenLDP}}\label{GenLDP}

First we prove two auxiliary lemmas. The first one states, informally speaking, that the measure $\tP_{\V a,\V p}$ is ``uniformly continuous'' in~$\V a$.

\begin{lemma}\label{lm:DifferentRatios}
	For every symmetric	matrix $\V p\in [0,1]^{k\times k}$, real $\e\in (0,1)$ and non-zero integer vectors $\V a,\V b\in\I N^k\NoZero$, if 
	$$\|\V b-\V a\|_1\le \e\, \min\left\{\,\|\V a\|_1,\|\V b\|_1\,\right\}$$ 
	then there is a (discrete) measure $\T{\I C}$ on $\tW\times \tW$ which gives a coupling between $\tP_{\V a,\V p}$ and $\tP_{\V b,\V p}$ such that for every $(\T U,\T V)$ in the support of $\T{\I C}$ we have $\delta_\Box(U,V)\le 4\e/(1-\e)$.
	\end{lemma}

\begin{proof}
	Let $m:=\|\V a\|_1$ and $n:=\|\V b\|_1$. 
%Assume without loss of generality that $m\le n$.
		Let $[m]=A_1\cup\dots\cup A_k$ (resp.\ $[n]=B_1\cup\dots\cup B_k$) be the partition into consecutive intervals with $a_1,\dots,a_k$ (resp.\ $b_1,\dots,b_k$) elements. For each $i\in [k]$ fix some subsets $A_i'\subseteq A_i$ and $B_i'\subseteq B_i$ of size $\min(a_i,b_i)$. Define $A':=\cup_{i=1}^k A_i'$ and
	$B':=\cup_{i=1}^k B_i'$. Fix any bijection $h:A'\to B'$ that sends each $A_i'$ to $B_i'$. We have
	\beq{eq:m-A'}
	\left|\,[m]\setminus A'\,\right|\le\sum_{i=1}^k \max(0,a_i-b_i)\le \|\V a-\V b\|_1\le \e m
	\eeq
	and similarly $\left|\,[n]\setminus B'\,\right|\le \e n$. 
	
We can couple random graphs $G\sim \I G(\V a,\V p)$ and $H\sim\I G(\V b,\V p)$ so that every pair $\{x,y\}$ in $A'$ is an edge in $G$ if and only if $\{h(x),h(y)\}$ is an edge in~$H$.
This is possible because, with $i,j\in [k]$ satisfying $(x,y)\in A_i'\times A_j'$, the probability of $\{x,y\}\in E(G)$ is $p_{i,j}$, the same as the probability with which $h(x)\in B_i$ and $h(y)\in B_j$ are made adjacent in~$H$ (so we can just use the same coin flip for both pairs). By making all edge choices to be mutually independent otherwise, we get a probability measure $\I C$ on pairs of graphs which is a coupling between $\I G(\V a,\V p)$ and $\I G(\V b,\V p)$.
 The corresponding measure $\T{\I C}$ on $\tW\times \tW$ gives  a coupling between $\tP_{\V a,\V p}$ and $\tP_{\V b,\V p}$. 

Let us show that this coupling $\I C$ satisfies the distance requirement. Take any pair of graphs $(G,H)$ in the support of~$\I C$. Let $G':=G[A']$ be obtained from $G$ by removing all vertices in $[m]\setminus A'$. We remove at most $\e$ fraction of vertices by~\eqref{eq:m-A'}. By relabelling the vertices of~$G$, we can assume that all removed vertices come at the very end. Then the union of the intervals in the graphon $\f{G}$ corresponding to the vertices of $G'$ is an initial segment of $[0,1]$ of length $s\ge 1-\e$ and the graphon of $G'$ is the pull-back of $\f{G}$ under the map $x\mapsto sx$.
By Lemma~\ref{lm:Delete01}, we have $\delta_\Box(\f{G},\f{G'})\le 2(\frac{1}{1-\e}-1)= \frac{2\e}{1-\e}$. By symmetry, the same estimate applies to the pair $(H,H')$, where $H'$ is obtained from $H$ by deleting all vertices from~$[n]\setminus B'$. By the definition of our coupling, the function $h:V(G')\to V(H')$ is (deterministically) an isomorphism between $G'$ and~$H'$. Thus the graphons of $G'$ and $H'$ are weakly isomorphic. 
This gives by the Triangle Inequality the required upper bound on the cut-distance between $\tf{G}$ and $\tf{H}$. 
\end{proof}

The function $J_{\V\alpha,\V p}(\T V)$ is not in general continuous in $\V\alpha$ even when $\V p$ and $\T V$ are fixed. (For example, if $k=2$, $V=1-\f{K_2}$ is the limit of $K_n\sqcup K_n$, the disjoint union of two cliques of order $n$ each, and $\V p$ is the identity $2\times 2$ matrix then $J_{\V\alpha,\V p}(\T V)$ is $0$ for $\V\alpha=(\frac12,\frac12)$ and $\infty$ otherwise.) However, the following version of ``uniform semi-continuity'' in $\V\alpha$ will suffice for our purposes.

\begin{lemma}\label{lm:ContOfJ}
	Fix any symmetric $\V p\in [0,1]^{k\times k}$.
	Then for every $\eta>0$ there is $\e=\e(\eta,\V p)>0$ such that if  vectors ${\V\gamma},\V\kappa\in [0,\infty)^k$ with $\|\V\gamma\|_1=\|\V\kappa\|_1=1$ satisfy
	\begin{equation}\label{eq:cond}
		\kappa_i\le (1+\e)\gamma_i,\quad\mbox{for every $i\in [k]$},
		%\mbox{ and if $\kappa_i=0$ then $\gamma_i=0$,}		
		\end{equation}
	 then
	for every $U\in {\mathcal{W}}$ there is ${V}\in {\C W}$ with $\delta_\Box(U,{V})\le \eta$ and
	$$J_{\V\kappa,\V p}(\T{V})\le J_{{\V\gamma},\V p}(\T U) +\eta.$$
\end{lemma}
\begin{proof}
%	Since $J_{\V\alpha,\V p}$ is unaffected when we remove zero entries from $\V\alpha$ (and update $k$ and the matrix $\V p$ accordingly), we can assume by~\eqref{eq:cond} that $\V\kappa,\V\gamma\in (0,\infty)^k$.
	For every fixed $p\in (0,1)$, the relative entropy function $h_p:[0,1]\to [0,\infty]$ is bounded (in fact, by  $\max\{h_p(0),h_p(1)\}<\infty$). Thus we can find a constant $C$ that depends on $\V p$ only such that for every $x\in [0,1]$ and $i,j\in [k]$, $h_{p_{i,j}}(x)$ is either $\infty$ or at most~$C$.

Let us show that any positive $\e\le \min\{\,\frac{\eta}{3(2C+\eta)},\,\frac{\eta}2\,\}$ works. Let ${\V\gamma}$, $\V\kappa$, and $U$ be as in the lemma. 	Assume that $J_{{\V\gamma},\V p}(\T U)<\infty$ for otherwise  we can trivially take ${V}:=U$. By the choice of $C$, we have that $J_{{\V\gamma},\V p}(\T U)\le C$. 

By replacing $U\in\C W$ with a weakly isomorphic graphon, assume that, for the partition $(\cI{{\V\gamma}}{i})_{i\in [k]}$ of $[0,1]$ into intervals, 
we have
	$$\frac{1}{2}\sum_{i,j\in [k]} \int_{\cI{{\V\gamma}}{i}\times \cI{{\V\gamma}}{j}} I_{p_{i,j}}(U) \dd\lambda^{\oplus 2}<J_{{\V\gamma},\V p}(\T U)+\frac{\eta}{2}.$$
	
Let $\phi:[0,1]\to [0,1]$ be the a.e.\ defined function  which  on each interval $\cI{{\V\kappa}}{i}$ of positive length is the increasing linear function that bijectively maps this interval onto $\cI{{\V\gamma}}{i}$. 
The 
%(unique up to a $\lambda$-null set) 
Radon-Nikodym derivative $D:=\frac{\dd \mu}{\dd\lambda}$ of the push-forward $\mu:=\phi_*\lambda$ of the Lebesgue measure $\lambda$ along $\phi$ assumes value
$\kappa_i/\gamma_i$ a.e.\ on~$\cI{{\V\gamma}}{i}$. (The union of $\cI{\V\gamma}{i}$ with $\gamma_i=0$ 
is a countable and thus null set,
%has obviously Lebesgue measure 0 
and can be ignored from the point of view of~$D$.)
	Let $V:=U^\phi$. It is a graphon by the first part of Lemma~\ref{lm:AlmostMP}. 
We have by the definition of $J_{\V\kappa,\V p}$ that
	\begin{equation*}
			J_{\V\kappa,\V p}(\T{V})\le  \ \frac{1}{2}\sum_{i,j\in [k]} \int_{\cI{{\V\kappa}}{i}\times \cI{{\V\kappa}}{j}} I_{p_{i,j}}(U^\phi) \ \dd \lambda^{\oplus 2} 
			=  \ \frac{1}{2}\sum_{i,j\in [k]} \frac{\kappa_i\kappa_j}{\gamma_i\gamma_j}\int_{\cI{{\V\gamma}}{i}\times \cI{{\V\gamma}}{j}} I_{p_{i,j}}(U) \dd\lambda^{\oplus 2}.
	\end{equation*}
 By $\frac{\kappa_i\kappa_j}{\gamma_i\gamma_j}\le (1+\e)^2\le 1+3\e$, this  is at most
 $$
 (1+3\e)\,\frac1{2}\sum_{i,j\in [k]} \int_{\cI{{\V\gamma}}{i}\times \cI{{\V\gamma}}{j}} I_{p_{i,j}}(U) \dd\lambda^{\oplus 2} 
\le (1+3\e)  \left(J_{{\V\gamma},\V p}(\T U)+\frac{\eta}2 \right).
% \le  \ J_{{\V\gamma},\V p}(\T U)+ \eta.	
 $$	
 By $J_{{\V\gamma},\V p}(\T U)\le C$ and our choice of $\e$, this in turn is at most
$$
J_{{\V\gamma},\V p}(\T U) + 3\e C+(1+3\e)\frac{\eta}2\le  \ J_{{\V\gamma},\V p}(\T U)+ \eta.
$$

Thus it remains to estimate the cut-distance between $U$ and~${V}=U^\phi$. Since the Radon-Nikodym derivative $D=\frac{\dd \phi_*\lambda}{\dd\lambda}$ satisfies $D(x)\le 1+\e$ for a.e.\ $x\in [0,1]$ by~\eqref{eq:cond}, we have by Lemma~\ref{lm:AlmostMP} that $\delta_\Box(U,{V})\le 2\e\le \eta$. Thus $V$ is as desired.\end{proof}

Now we are ready to prove Theorem~\ref{th:GenLDP}.

\begin{proof}[Proof of Theorem~\ref{th:GenLDP}]
Recall that we have to establish an LDP for the sequence $(\tP_{\V a_n,\V p})_{n\in\I N}$ of measures, where the integer vectors $\V a_n$, after scaling by $1/n$, converge to a non-zero vector $\V\alpha\in[0,\infty)^k\NoZero$. Since the underlying metric space $(\tW,\delta_{\Box})$ is compact and the proposed rate function $J_{\V\alpha,\V p}$ is lsc by Theorem~\ref{th:JLSC}, it suffices to prove the bounds in~\eqref{eq:lower} and~\eqref{eq:upper} of Lemma~\ref{lm:LDP} for any given graphon~$\T U\in \tW$.

Let us show the upper bound first, that is, that
\beq{eq:red9}
 \lim_{\eta\to 0}\limsup_{n\to\infty} \frac{1}{(\|\V a_n\|_1)^2}\log \tP_{\V a_n,\V p}(\Sball(\widetilde U,\eta))\le -J_{\V\alpha,\V p}(\T U).
\eeq

 Take any $\eta>0$. Assume that, for example, $\eta<1/9$. Let $m\in\I N$ be sufficiently large. Define
\beq{eq:b}
 \V b:=(a_{m,1}\I 1_{\alpha_1>0},\ldots,a_{m,k}\I 1_{\alpha_k>0})\in\I Z^k.
 \eeq
 In other words, we let $\V b$ to be $\V a_m$ except we set its $i$-th entry $b_i$ to be~0  for each $i\in [k]$ with $\alpha_i=0$.
 By Theorem~\ref{th:BCGPS} (the LDP by Borgs et al~\cite{BCGPS}) and Theorem~\ref{th:JLSC}, we have that
 \beq{eq:BCGPSApplied1}
 \limsup_{n\to\infty} \frac{1}{(n\,\|\V b\|_1)^2}\log\tP_{n\V b,\V p}(\Sball(\widetilde U,\eta))\le 
 %-\inf\{J_{\V b,\V p}(V)\mid V\in\Sball(\widetilde U,\eta)\}.
 -\inf_{\Sball(\widetilde U,\eta)} J_{\V b,\V p}.
 \eeq

Since $m$ is sufficiently large, we can assume that $\|\frac1n\, \V a_n-\V\alpha\|_1\le \xi\,\|\V\alpha\|_1$ for every $n\ge m$, where e.g.~$\xi:=\eta/40$. 
In particular, we have 
$\|\V a_m-\V b\|_1
\le \|\V a_m-m\V\alpha\|_1
\le \xi m\|\V\alpha\|_1$ from which it follows that 
$$
 \|\V b -m\V\alpha\|_1\le \|\V b -\V a_m\|_1+\| \V a_m - m\V\alpha\|_1\le 2\xi m\,\|\V\alpha\|_1.
$$ 
Let $n$ be sufficiently large (in particular, $n\ge m$) and let $n':= \floor{n/m}$. By above, we have
\hide{\begin{eqnarray}
\|n'\V b-\V a_n\|_1&\le& \textstyle 
 \|\V b\|_1+\|\frac{n}{m}\V b-n\V\alpha\|_1+\|n\V\alpha-\V a_n\|_1\
\le\ \|\V b\|_1+3\xi n\|\V\alpha\|_1\nonumber\\
 &\le &\textstyle
 \|\V b\|_1 + 4\xi \min\{\,\|\frac{n}{m}\V b\|_1,\,\| \V a_n\|_1\,\}\ \le\ \frac{\eta}9\,\min\{\,\|n'\V b\|_1,\,\| \V a_n\|_1\,\}.\label{eq:n'b-an}
\end{eqnarray}
}
\begin{eqnarray}
\|n'\V b-\V a_n\|_1&\le& \textstyle 
 \|\V b\|_1+\|\frac{n}{m}\V b-n\V\alpha\|_1+\|n\V\alpha-\V a_n\|_1\nonumber\\
&\le& \|\V b\|_1+3\xi n\|\V\alpha\|_1\nonumber\\
 &\le &\textstyle
 \|\V b\|_1 + 4\xi \min\{\,\|\frac{n}{m}\V b\|_1,\,\| \V a_n\|_1\,\}\nonumber\\ 
&\le& \textstyle\frac{\eta}9\,\min\{\,\|n'\V b\|_1,\,\| \V a_n\|_1\,\}.\label{eq:n'b-an}
\end{eqnarray}
 Thus, by Lemma~\ref{lm:DifferentRatios}, there is a coupling $\T{\I C}$ between $\tP_{n'\V b,\V p}$ and $\tP_{\V a_n,\V p}$ such that for every $(\widetilde{V},\widetilde{W})$ in the support of $\T{\I C}$ we have $\delta_\Box(V,{W})\le \eta/2$. Thus, if $\widetilde{W}\sim\tP_{\V a_n,\V p}$ lands in $\Sball(\widetilde{ U},\eta/2)$ then necessarily $\widetilde{V}\sim \tP_{n'\V b,\V p}$ lands in $\Sball(\widetilde{U},\eta)$. This gives that
 $$
 \tP_{\V a_n,\V p}(\Sball(\widetilde U,\eta/2))\le \tP_{n'\V b,\V p}(\Sball(\widetilde U,\eta)).
 $$
 Since this is true for every sufficiently large $n$ and it holds by~\eqref{eq:n'b-an} that, for example, 
 $\|\V a_n\|_1/(n'\,\|\V b\|_1) \ge 1-\eta/9\ge \sqrt{1-\eta}$, we have that
 \beq{eq:red1}
 \limsup_{n\to\infty} \frac{1-\eta}{(\|\V a_n\|_1)^2}\log \tP_{\V a_n,\V p}(\Sball(\widetilde U,\eta/2))\le \limsup_{n'\to\infty}  \frac{1}{(n'\,\|\V b\|_1)^2}\log \tP_{n'\V b,\V p}(\Sball(\widetilde U,\eta)).
 \eeq

Let us turn our attention to the right-hand side of~\eqref{eq:BCGPSApplied1}. 
Pick some $\T V\in \Sball(\T U,\eta)$ with $J_{\V b,\V p}(\T V)\le \inf_{\Sball(\T U,\eta)} J_{\V b,\V p}+\eta$. (In fact, by the lower semi-continuity of $J_{\V b,\V p}$ and the compactness of $\Sball(\T U,\eta)$, we could have required
that $J_{\V b,\V p}(\T V)= \inf_{\Sball(\T U,\eta)} J_{\V b,\V p}$.) 
Since $\V a_n/\|\V a_n\|_1$ converges to $\V\alpha/\|\V\alpha\|_1$ as $n\to\infty$ and we chose $m$ to be sufficiently large,
we can assume that $b_i=0$ if and only if $\alpha_i=0$ and that 
Lemma~\ref{lm:ContOfJ} applies for $\V\gamma:=\V b/\|\V b\|_1$ and $\V\kappa:=\V\alpha/\|\V\alpha\|_1$ (and our~$\eta$). 
The lemma gives that 
there is $\T{W}\in\Sball(\T V,\eta)$ such that $J_{\V\alpha,\V p}(\T{W})-\eta\le J_{\V b,\V p}(\T{V})$. Thus we get the following upper bound on the right-hand side of~\eqref{eq:BCGPSApplied1}:
\beq{eq:red2}
-\inf_{\Sball(\T U,\eta)} J_{\V b,\V p}\le -J_{\V b,\V p}(\T V)+\eta\le -J_{\V\alpha,\V p}(\T{W})+2\eta\le -\inf_{\Sball(\T U,2\eta)} J_{\V \alpha,\V p}+2\eta.
\eeq
 By putting~\eqref{eq:BCGPSApplied1}, \eqref{eq:red1} and~\eqref{eq:red2} together we get that, for every $\eta\in (0,1/9)$,
 $$
 \limsup_{n\to\infty} \frac{1-\eta}{(\|\V a_n\|_1)^2}\log \tP_{\V a_n,\V p}(\Sball(\widetilde U,\eta/2))\le
-\inf_{\Sball(\T U,2\eta)} J_{\V \alpha,\V p}+2\eta.
$$ 
 If we take here the limit as $\eta\to 0$ then the infimum in the right-hand side converges to~$J_{\V \alpha,\V p}(\T U)$ by the lower semi-continuity of $J_{\V \alpha,\V p}$ (established in Theorem~\ref{th:JLSC}), giving the claimed upper bound~\eqref{eq:red9}.

Let us turn to the lower bound, i.e.\ we prove~\eqref{eq:lower} for $\T U\in\tW$. As before, take any sufficiently small $\eta>0$, then sufficiently large $m\in\I N$ and define $\V b$ by~\eqref{eq:b}. By Theorems~\ref{th:BCGPS} and~\ref{th:JLSC} applied to the open ball around $\T U$ of radius $2\eta$, we have
 \beq{eq:red4}
\liminf_{n\to\infty} \frac{1}{(n\,\|\V b\|_1)^2}\log\tP_{n\V b,\V p}(\Sball(\widetilde U,2\eta))\ge 
%-\inf\{J_{\V b,\V p}(V)\mid V\in\Sball(\T U,\eta)\}.
-\inf_{\Sball(\T U,\eta)} J_{\V b,\V p}.
 \eeq
 Similarly as for the upper bound, the left-hand side can be upper bounded via Lemma~\ref{lm:ContOfJ} by, for example, 
 \beq{eq:red3}
 \liminf_{n\to\infty} \frac{1+\eta}{(\|\V a_n\|_1)^2}\log \tP_{\V a_n,\V p}(\Sball(\widetilde U,3\eta)). 
 \eeq
 Since $m$ is sufficiently large,
 we can apply Lemma~\ref{lm:ContOfJ} to $\V\kappa:=\B b/\|\V b\|_1$, $\V\gamma:=\V\alpha/\|\V\alpha\|_1$, the given graphon $U$ and our chosen $\eta$ 
 %(and our $\V p$) 
 to find $\T{V}\in \Sball(\T U, \eta)$ such that $J_{\V b,\V p}(\T{V})\le J_{\V\alpha,\V p}(\T U)+\eta$. Thus
 $$
 -\inf_{\Sball(\T U,\eta)} J_{\V b,\V p}\ge -J_{\V b,\V p}(\T{V})\ge -J_{\V\alpha,\V p}(\T U)-\eta.
 $$
 By~\eqref{eq:red4}, this is a lower bound on the expression in~\eqref{eq:red3}. Taking the limit as $\eta\to 0$ we get the required LDP lower bound~\eqref{eq:lower}. This finishes the proof of Theorem~\ref{th:GenLDP}.
 \end{proof}

\section{Proof of Theorem~\ref{th:ourLDP}}\label{ourLDP}

Recall that $W$ is a $k$-step graphon with non-null parts whose values are encoded by a symmetric $k\times k$ matrix $\V p\in [0,1]^{k\times k}$. We consider the $W$-random graph $\I G(n,W)$ where we first sample $n$ independent uniform points $x_1,\dots,x_n\in [0,1]$ and then make each pair $\{i,j\}\subseteq [n]$ an edge with probability $W(x_i,x_j)$. We have to prove an LDP for the corresponding sequence $(\tR_{n,W})_{n\in\I N}$ of measures on the metric space $(\tW,\delta_\Box)$ with speed $n^2$ and the rate function $R_{\V p}$ that was defined in~\eqref{eq:R}. Recall that the lower semi-continuity of $R_{\V p}$ was established in Corollary~\ref{cr:lsc}.

Let us show the lower bound. Since the underlying space $(\tW,\delta_\Box)$ is compact, it is enough to prove the bound in~\eqref{eq:lower} of Lemma~\ref{lm:LDP} for any given $\T U\in \T{\C W}$, that is, that
\beq{eq:L1}
\lim_{\eta\to 0}\liminf_{n\to\infty} \frac{1}{n^2} \log \tR_{n,W}(\Sball(\widetilde U,\eta))\ge - R_{\V p}(\T U).
\eeq
Take any $\e>0$. By the definition of $R_{\V p}$, we can fix a vector $\V\alpha=(\alpha_1,\dots,\alpha_k)\in [0,1]^k$ 
such that $\|\V\alpha\|_1=1$ and 
 $$
 R_{\V p}(\T U)\ge J_{\V \alpha,\V p}(\T U)-\e.
$$
Let  $m$ be the number of non-zero entries of $\V \alpha$. As $\V \alpha$ is non-zero, we have $m\ge 1$. We can assume by symmetry that $\alpha_1,\dots,\alpha_m$ are the non-zero entries. Define $c:=\frac12 \min_{i\in [m]} \alpha_i>0$. 

For each $n\in\I N$, take any integer vector $\V a_n=(a_{n,1},\dots,a_{n,k})\in \NZ^k$ such that $\|\V a_n\|_1=n$ and $\|\V a_{n}-n\,\V\alpha\|_\infty<1$ (in particular, we have $a_{n,i}=0$ if $\alpha_i=0$). 

Let $n$ be sufficiently large. In particular, for every $i\in [m]$ we have that $a_{n,i}/n\ge c$. 
When we generate $G\sim \I G(n,W)$ by choosing first random $x_1,\dots,x_n\in [0,1]$, 
 it holds with probability at least, very roughly, $c^{-n}$ that for each $i\in [k]$ the number of $x_j$'s that belong to the $i$-th part of the step graphon~$W$ is exactly~$a_{n,i}$. Conditioned on this event of positive measure, the resulting graphon
 $\tf{G}$ is distributed according to $\tP_{\V a_n,\V p}$. Thus 
 $$
 \tR_{n,W}(S)\ge c^{-n}\,\tP_{\V a_n,\V p}(S),\quad\mbox{for every $S\subseteq \tW$}.
 $$
 This and the new LDP result (Theorem~\ref{th:GenLDP}) give that, for every $\eta>0$, %we have with $S:=\Sball(\widetilde U,\eta)$ that
 \begin{eqnarray*}
 \limsup_{n\to\infty} \frac{1}{n^2} \log \tR_{n,W}(\Sball(\widetilde U,\eta))&\ge& \limsup_{n\to\infty} \frac{1}{n^2} \log \tP_{\V a_n,\V p}(\Sball(\widetilde U,\eta))\\
 &\ge& -\inf_{\Sball(\T U,\eta/2)} J_{\V\alpha,\V p}
 \ \ge\ -J_{\V\alpha,\V p}(\T U)\ \ge\ -R_{\V p}(\T U)-\e.
 \end{eqnarray*}
 Taking the limit as $\eta\to 0$, we conclude that the LDP lower bound~\eqref{eq:L1} holds within additive error~$\e$. As $\e>0$ was arbitrary, the lower bound
 % in~\eqref{eq:L1} 
 holds.
 
 Let us show the upper bound~\eqref{eq:upperGen} of Definition~\ref{df:LDP} for any closed set $F\subseteq \tW$, that is, that
  \beq{eq:U3}
 \limsup_{n\to\infty} \frac1{n^2} \log \tR_{n,W}(F)\le -\inf_F R_{\V p}.
 \eeq

 For each $n\in\I N$, we can write $\tR_{n,W}(F)$ as the sum over all  $\V a\in\NZ^k$ with $\|\V a\|_1=n$ of the probability that the distribution of random independent $x_1,\dots,x_n\in [0,1]$ per $k$ steps of $W$ is given by the vector $\V a$ times the probability conditioned on $\V a$ to hit the set~$F$. This conditional probability is exactly $\tP_{\V a,\V p}(F)$. Thus $\tR_{n,W}(F)$ is a convex combination of the reals $\tP_{\V a,\V p}(F)$ and there is a vector $\V a_n\in \NZ^k$ with $\|\V a_n\|_1=n$ such that
 \beq{eq:an}
 \tR_{n,W}(F)\le \tP_{\V a_n,\V p}(F).
 \eeq
 Fix one such vector~$\V a_n$ for each $n\in\I N$.
 
 Since the set of all real vectors $\V\alpha\in [0,1]^k$ with $\|\V\alpha\|_1=1$ is compact, we can find an increasing sequence $(n_i)_{i\in\I N}$ of integers such that 
 \beq{eq:U1}
 \limsup_{n\to\infty} \frac1{n^2} \log \tR_{n,W}(F)=\lim_{i\to\infty} \frac1{n_i^2} \log \tR_{n_i,W}(F),
 \eeq
 and the scaled vectors $\frac1{n_i}\,\V a_{n_i}$ converge to some real vector $\V\alpha\in [0,1]^k$.
 
 Let $(\V b_n)_{n\in\I N}$ be obtained by filling the gaps in $(\V a_{n_i})_{n\in\I N}$, meaning that if $n=n_i$ for some  $i\in\I N$ then we let $\V b_n:=\V a_{n_i}$; otherwise we pick any $\V b_n\in\NZ^k$ that satisfies $\|\V b_n\|_1=n$ and $\|\V b_n-n\,\V\alpha\|_\infty<1$. Since the normalised vectors $\V b_{n}/\|\V b_{n}\|_1$ converge to the same limiting vector~$\V\alpha$, we have by Theorem~\ref{th:GenLDP} that
 \beq{eq:U2}
 \limsup_{i\to\infty} \frac1{n_i^2} \log \tP_{\V a_{n_i},\V p}(F)\le \limsup_{n\to\infty} \frac1{n^2} \log \tP_{\V b_{n},\V p}(F)\le -\inf_F J_{\V\alpha,\V p}.
 %\le -\inf_F R_{\V p}.
 \eeq
 Putting~\eqref{eq:an}, \eqref{eq:U1} and~\eqref{eq:U2} together with the trivial consequence $J_{\V\alpha,\V p}\ge R_{\V p}$ of the definition of $R_{\V p}$, we get the desired upper bound~\eqref{eq:U3}. This finishes the proof of Theorem~\ref{th:ourLDP}.

\section*{Acknowledgements}

Jan Greb\'\i k and Oleg Pikhurko were supported by 
Leverhulme Research Project Grant RPG-2018-424.

%\bibliography{oleg,sets,misc,ramsey,enum,number,posets,sat,ex,matroid,design,random,graph,general,geometry,algorithm,Analysis,limits}

\begin{bibdiv}
\begin{biblist}

\bib{BCGPS}{unpublished}{
      author={Borgs, C.},
      author={Chayes, J.},
      author={Gaudio, J.},
      author={Petti, S.},
      author={Sen, S.},
       title={A large deviation principle for block models},
        date={2020},
        note={E-print arxiv:2007.14508},
}

\bib{Chatterjee16bams}{article}{
      author={Chatterjee, S.},
       title={An introduction to large deviations for random graphs},
        date={2016},
     journal={Bull. Amer. Math. Soc. (N.S.)},
      volume={53},
       pages={617\ndash 642},
}

\bib{Chatterjee17ldrg}{book}{
      author={Chatterjee, S.},
       title={Large deviations for random graphs},
      series={Lecture Notes in Mathematics},
   publisher={Springer, Cham},
        date={2017},
      volume={2197},
        note={Lecture notes from the 45th Probability Summer School held in
  Saint-Flour, June 2015},
}

\bib{ChatterjeeVaradhan11}{article}{
      author={Chatterjee, S.},
      author={Varadhan, S. R.~S.},
       title={The large deviation principle for the {Erd\H{o}s-R\'enyi} random
  graph},
        date={2011},
     journal={Europ.\ J.\ Combin.},
      volume={32},
       pages={1000\ndash 1017},
}

\bib{Cohn13mt}{book}{
      author={Cohn, D.~L.},
       title={Measure theory},
     edition={Second},
      series={Birkh\"{a}user Advanced Texts: Basel Textbooks},
   publisher={Birkh\"{a}user/Springer, New York},
        date={2013},
}

\bib{DemboZeitouni10ldta}{book}{
      author={Dembo, A.},
      author={Zeitouni, O.},
       title={Large deviations techniques and applications},
      series={Stochastic Modelling and Applied Probability},
   publisher={Springer-Verlag, Berlin},
        date={2010},
      volume={38},
        note={Corrected reprint of the second (1998) edition},
}

\bib{GrebikPikhurko:LDP}{unpublished}{
      author={Greb{\'\i}k, J.},
      author={Pikhurko, O.},
       title={Large deviation for graphon sampling},
        date={2021},
        note={In preparation},
}

\bib{Kechris:cdst}{book}{
      author={Kechris, A.~S.},
       title={Classical descriptive set theory},
      series={Graduate Texts in Mathematics},
   publisher={Springer-Verlag, New York},
        date={1995},
      volume={156},
}

\bib{Lovasz:lngl}{book}{
      author={{Lov\'asz}, L.},
       title={Large networks and graph limits},
      series={Colloquium Publications},
   publisher={Amer.\ Math.\ Soc.},
        date={2012},
}

\bib{LovaszSzegedy07gafa}{article}{
      author={Lov{\'a}sz, L.},
      author={Szegedy, B.},
       title={{Szemer\'edi's} lemma for the analyst},
        date={2007},
     journal={Geom.\ Func.\ Analysis},
      volume={17},
       pages={252\ndash 270},
}

\bib{RassoulaghaSeppelainen14cldigm}{book}{
      author={Rassoul-Agha, F.},
      author={Sepp\"{a}l\"{a}inen, T.},
       title={A course on large deviations with an introduction to {G}ibbs
  measures},
      series={Graduate Studies in Mathematics},
   publisher={American Mathematical Society, Providence, RI},
        date={2015},
      volume={162},
}

\bib{Srivastava98cbs}{book}{
      author={Srivastava, S.~M.},
       title={A course on {B}orel sets},
      series={Graduate Texts in Mathematics},
   publisher={Springer-Verlag, New York},
        date={1998},
      volume={180},
}

\end{biblist}
\end{bibdiv}


\end{document}